\documentclass[12pt]{amsart}

\usepackage{fix-cm}
\usepackage[normalem]{ulem}

\usepackage{lineno}

\usepackage{soul}

\usepackage{graphicx}

\usepackage{amssymb, amsmath, amsthm, hyperref, euscript, color}

\usepackage[dvipsnames]{xcolor}

\usepackage{amscd}

\usepackage{tikz-cd}

\usepackage{bbm}

\usepackage{enumitem}

\usepackage[english]{babel}

\usepackage{mathtools}

\usepackage{makecell}

\numberwithin{equation}{section}

\newtheoremstyle{theor}{6pt plus 1pt minus 1pt}{6pt plus 1pt minus 1pt}{\slshape}{}{\bfseries}{.}{5pt plus 1pt minus 1pt}{}
\newtheoremstyle{def}{6pt plus 1pt minus 1pt}{6pt plus 1pt minus 1pt}{}{}{\bfseries}{.}{5pt plus 1pt minus 1pt}{}
\newtheoremstyle{rmk}{6pt plus 1pt minus 1pt}{6pt plus 1pt minus 1pt}{}{}{\bfseries}{.}{5pt plus 1pt minus 1pt}{}
\newtheoremstyle{claim}{6pt plus 1pt minus 1pt}{6pt plus 1pt minus 1pt}{}{}{\bfseries}{.}{5pt plus 1pt minus 1pt}{}

\theoremstyle{theor}
\newtheorem{newstatement}{newstatement}

\newtheorem{theorem}[newstatement]{Theorem}
\newtheorem*{theorem*}{Theorem 2}

\newtheorem{proposition}[newstatement]{Proposition}

\theoremstyle{def}

\theoremstyle{rmk}

\newtheorem*{example*}{Example}

\theoremstyle{claim}

\theoremstyle{theor}
\newtheorem{thm}{Theorem}

\expandafter\let\expandafter\oldproof\csname\string\proof\endcsname
\let\oldendproof\endproof
\renewenvironment{proof}[1][\proofname]{%
  \oldproof[\slshape #1]%
}{\oldendproof}
\def\provedboxcontents#1{$\square$}

\makeatletter
\newsavebox\myboxA
\newsavebox\myboxB
\newlength\mylenA

\newcommand*\xoverline[2][0.75]{%
    \sbox{\myboxA}{$\m@th#2$}%
    \setbox\myboxB\null
    \ht\myboxB=\ht\myboxA%
    \dp\myboxB=\dp\myboxA%
    \wd\myboxB=#1\wd\myboxA
    \sbox\myboxB{$\m@th\overline{\copy\myboxB}$}
    \setlength\mylenA{\the\wd\myboxA}
    \addtolength\mylenA{-\the\wd\myboxB}%
    \ifdim\wd\myboxB<\wd\myboxA%
       \rlap{\hskip 0.5\mylenA\usebox\myboxB}{\usebox\myboxA}%
    \else
        \hskip -0.5\mylenA\rlap{\usebox\myboxA}{\hskip 0.5\mylenA\usebox\myboxB}%
    \fi}
\makeatother

\newcommand{\Z}{\mathbb{Z}}

\newcommand{\N}{\mathbb{N}}

\newcommand{\RP}{{\mathbb R\mkern-0.5mu\mathrm P}}

\DeclareMathOperator{\Aut}{Aut}

\DeclareMathOperator{\Pin}{Pin}
\DeclareMathOperator{\Def}{def}

\DeclareMathOperator{\Sp}{Spin}

\DeclareMathOperator{\So}{SO}

\DeclareMathOperator{\Rk}{rk}
\newcommand{\cs}{\mathbin{\#}}

\makeatletter
\@namedef{subjclassname@2020}{%
  \textup{2020} Mathematics Subject Classification}
\makeatother

\begin{document}

\author{Rafael Torres}

\title[Exotic non-orientable four-manifolds with prescribed $\pi_1$.]{Exotic non-orientable four-manifolds with prescribed fundamental group.}

\address{Scuola Internazionale Superiori di Studi Avanzati (SISSA)\\ Via Bonomea 265\\34136\\Trieste\\Italy}

\email{rtorres@sissa.it}


\maketitle

\emph{Abstract}: We show that any finitely presented group with an index two subgroup is realized as the fundamental group of a closed smooth non-orientable four-manifold that admits an exotic smooth structure, which is obtained by performing a Gluck twist. The orientation 2-covers of these four-manifolds are diffeomorphic. These two smooth structures remain inequivalent after adding arbitrarily many copies of the product of a pair of 2-spheres and stabilize after adding a single copy of the complex projective plane.

\section{Introduction.}\label{Introduction}

Kreck showed \cite[Theorem 1]{[Kreck]} that for any finitely presented group $\pi$ with a non-trivial element $w_1\in H^1(\pi; \Z/2)$, there is a closed smooth four-manifold $M_{\pi}$ with fundamental group $\pi$ and orientation-character $w_1$ such that the connected sums $M_{\pi}\cs K3$ and $M_{\pi}\cs11(S^2\times S^2)$ are homeomorphic, but not diffeomorphic (see \cite[Section 5.5]{[KasprowskiPowellRay]}). The K3-surface is denoted by $K3$ \cite[Section 1.3]{[GompfStipsicz]}. The purpose of this brief note is to strengthen and extend Kreck's result as follows.


\begin{thm}\label{Theorem A}Let $G$ be a finitely presented group that has a subgroup of index two and denote its deficiency by $\Def(G)$. There are closed smooth non-orientable four-manifolds $M_G$ and $N_G$ with fundamental group $G$ and that satisfy the following properties.\begin{itemize}
\item The Euler characteristic of $M_G$ is $\chi(M_G) = 4 - 2\Def (G)$.
\item There is a homeomorphism\begin{equation*}M_G\rightarrow N_G.\end{equation*}
\item There is no diffeomorphism\begin{equation*}M_G\cs(n - 1)(S^2\times S^2)\rightarrow N_G\cs(n - 1)(S^2\times S^2)\end{equation*}for any $n\in \N$.
\item There is a diffeomorphism\begin{equation*}\widehat{M}_G\rightarrow \widehat{N}_G\end{equation*}between their orientation 2-covers $\widehat{M}_G\rightarrow M_G$ and $\widehat{N}_G\rightarrow N_G$, which admit a $\Sp$ structure.
\item The four-manifold $N_G$ is obtained by performing a Gluck twist along a smoothly embedded 2-sphere $S\hookrightarrow M_G$.
 \item There is a diffeomorphism\begin{equation*}M_G\cs \mathbb{CP}^2\rightarrow N_G\cs \mathbb{CP}^2.\end{equation*}
\end{itemize}

\end{thm}

The first and last three items of Theorem \ref{Theorem A} do not appear in Kreck's work. An advantage of the construction used in the proof of Theorem \ref{Theorem A} over the examples of Kreck is that the Euler characteristic of the four-manifolds considered here is considerably smaller. For the sake of completeness, we show in Section \ref{Section Kreck} that the four-manifolds $M_{\pi}\cs K3$ and $M_{\pi}\cs 11(S^2\times S^2)$ also satisfy the properties listed in the fourth and sixth items of the theorem. Cappell-Shaneson \cite{[CappellShaneson]} showed that if $X$ is a closed smooth four-manifold such that its fundamental group contains an orientation-reversing element of order two, then there is a closed smooth non-orientable four-manifold $Q$ and a simple homotopy equivalence $f: Q\rightarrow X$ whose smooth normal invariant is the nontrivial element in the kernel of $[X; G/O]\rightarrow [X, G/TOP]$. The four-manifold $Q$, however, might well be diffeomorphic to $X$ as shown by Akbulut in the case $X = S^2\times \RP^2$ \cite{[Akbulut1]}. We use the $\eta$-invariant of the twisted Dirac operator to show that the four-manifolds of Theorem \ref{Theorem A} are not diffeomorphic through work of Stolz; see Section \ref{Section Invariant}. The construction of the exotic smooth structure $N_G$ of Theorem \ref{Theorem A} uses an exotic smooth structure $A$ due to Akbulut \cite{[Akbulut1], [Akbulut2]} that arises from performing a Gluck twist to the connected sum $(S^3\widetilde{\times} S^1)\cs(S^2\times S^2)$ of the non-orientable 3-sphere bundle over the circle $S^3\widetilde{\times} S^1$ with a copy of $S^2\times S^2$; see Section \ref{Section Akbulut Torres}. A result of Gompf \cite{[Gompf]} implies that the orientation 2-cover of $A$ is diffeomorphic to $(S^1\times S^3) \cs 2(S^2\times S^2)$ as pointed out by Torres in \cite[Proposition 9]{[Torres]}. Hence, Theorem \ref{Theorem A} provides a multitude of new examples of exotic orientation-reversing free involutions on four-manifolds \cite{[CappellShaneson], [FintushelStern1], [Gompf], [Torres]}. 

Kreck's examples and the four-manifolds of Theorem \ref{Theorem A} satisfy the identity $w_1^2 = w_2$, where $w_i$ is the ith Stiefel-Whitney class. Recently, Kasprowski-Powell \cite{[KasprowskiPowell]} used Kreck's construction to produce four-manifolds that satisfy the second and third items of Theorem \ref{Theorem A} and for which $w_1^2\neq w_2$.

\subsection{Acknowledgements} We thank Daniel Kasprowski and the referee for pointing out several inaccuracies in an earlier version of the note, and their input on how to correct them. In particular, Daniel Kasprowski suggested the argument used to prove Theorem \ref{Theorem 1}. We thank Valentina Bais for her suggestions to improve the exposition. Zahvaljujemo se ekipi PRH v Tolminu, Kobaridu in Ladri za njihovo velikodušno gostoljubje med pisanjem tega članka.




\section{A spectral invariant to distinguish smooth structures.}\label{Section Invariant}

This section contains a brief discussion on $\Pin^+$-structures and on the $\eta$-invariant of a closed smooth four-manifold that is equipped with such a structure. The reader is directed towards \cite{[KirbyTaylor], [Stolz]} for background results. In the sequel, a $\Pin^+$-structure $\phi_M$ on a smooth four-manifold $M$ will be denoted by $(M, \phi_M)$. We also refer to the pair $(M, \phi_M)$ as a $\Pin^+$-manifold. For any such four-manifold $M$ with first cohomology group $H^1(M; \Z/2) = \Z/2$, there are two $\Pin^+$-structures $\{(M, \pm \phi_M)\}$, and their classes $\{[(M, \pm \phi_M)]\}$ are mutual inverses in the fourth $\Pin^+$-bordism group $\Omega^{\Pin^+}_4 = \Z/16$ \cite[Corollary 6.4]{[Stolz]}, \cite[p. 190, Theorem 5.2]{[KirbyTaylor]}. This group is generated by the real projective 4-space \cite[Theorem 5.2]{[KirbyTaylor]}, and $\eta(\RP^4, \pm \phi_{\RP^4}) = \pm \frac{1}{8} \mod 2\Z$ \cite[Corollary 5.4]{[Stolz]}. The following theorem is paramount for our purposes.

\begin{theorem}\label{Theorem Stolz}Stolz \cite{[Stolz]}. Let $(M, \phi_M)$ be a closed smooth $\Pin^+$- four-manifold.

$\bullet$ The invariant $\eta(M, \phi_M) \mod 2\Z$ is a complete $\Pin^+$-bordism invariant.

$\bullet$ If $M$ is simply connected, then $\eta(M, \phi_M) = \frac{1}{16}\sigma(M) \mod 2\Z$, where $\sigma(M)$ is the signature. 


\end{theorem}

 Stolz used these properties along with the modified surgery theory methods developed by Kreck \cite{[Kreck], [Kreck1]} to show that Cappell-Shaneson's exotic $\RP^4$ is not stably diffeomorphic to the standard one \cite[Theorems B and 7.4]{[Stolz]}. By building heavily on these results, we record the following generalization of Stolz's result.

\begin{theorem}\label{Theorem Stolz Main}Let $M_1$ and $M_2$ be closed non-orientable smooth four-manifolds that admit a $\Pin^+$-structure and that have the same Euler characteristic. If for every $\Pin^+$-structure $\{(M_i, \phi_{M_i}): i = 1, 2\}$, the values of the corresponding $\eta$-invariants satisfy\begin{equation}\label{Different Values}\eta(M_1, \phi_{M_1}) \neq \eta(M_2, \phi_{M_2}) \mod 2\Z,\end{equation}then there is no diffeomorphism\begin{equation}M_1\cs(n - 1)(S^2\times S^2) \rightarrow M_2\cs(n - 1)(S^2\times S^2)\end{equation} for any $n\in \N$.
\end{theorem}

\begin{proof}Recall that a normal 1-type $\xi: B\rightarrow BO$ of a four-manifold $M$ determines a bordism theory $\Omega^\xi_4$ of four-manifolds that admit a lift of their stable normal bundle across the fibration $\xi$ \cite{[Kreck], [Kreck1]}. A choice of normal 1-smoothing determines a class in the group $\Omega^\xi_4$. Different normal 1-smoothings need not define the same class in $\Omega^\xi_4$. The main ingredient in the proof of Theorem \ref{Theorem Stolz Main} is the following result.\begin{theorem}\label{Theorem Kreck}Kreck \cite[Theorem C]{[Kreck1]}. Let $M_1$ and $M_2$ be two closed smooth four-manifolds with the same Euler characteristic and that share the same normal 1-type $\xi: B\rightarrow BO$. The four-manifolds $M_1$ and $M_2$ admit normal 1-smoothings that are bordant in $\Omega^\xi_4$, if and only if there is an $n\in \N$ such that the connected sums $M_1\cs (n_1 - 1)(S^2\times S^2)$ and $M_2\cs(n_2 - 1)(S^2\times S^2)$ are diffeomorphic. \end{theorem}

We now argue that Theorem \ref{Theorem Stolz Main} is a corollary of Theorem \ref{Theorem Kreck}. As it was observed by Kasprowski-Land-Powell-Teichner \cite[Theorem 1.1]{[KasprowskiLandPowellTeichner]}, there is a one-to-one correspondence between stable diffeomorphism classes of 4-manifolds with the same normal 1-type and $\Omega^\xi_4/\Aut(\xi)$. In the presence of a $\Pin^+$-structure, there is a map $\Omega^\xi_4 \rightarrow \Omega^{\Pin^+}_4\cong \Z/16$ \cite{[Giambalvo], [KirbyTaylor]} and $\Aut(\xi) = \Z/2$ \cite[\S 5]{[Kreck]}. We thank the referee for pointing this out. Moreover, the four-manifolds $M_1$ and $M_2$ of Theorem \ref{Theorem Stolz Main} do not have the same normal 1-type since they both admit a $\Pin^+$-structure \cite[\S 2 Proposition 2]{[Kreck1]}, and the $\eta$-invariant detects the nine equivalence classes\begin{center}$\{0, \pm 1, \pm 2, \pm 3\pm 4, \pm 5, \pm 6, \pm 7, 8\} \in \Omega^{\Pin^+}_4/(x\sim -x)$.\end{center}\end{proof}

\section{Proof of Theorem \ref{Theorem A}.}\label{Section Proof} 

\subsection{A result of Akbulut}\label{Section Akbulut Torres} For the sake of clarity in the proof of Theorem \ref{Theorem A}, we now summarize the result of Akbulut and the observations on it due to Torres that were mentioned in the introduction. A Gluck twist along a smoothly embedded 2-sphere $S\subset X$ with tubular neighborhood $\nu(S) = D^2\times S^2$ is the cut-and-paste construction $X_S = (X\setminus \nu(S))\cup_{\varphi} (D^2\times S^2)$, where the gluing diffeomorphism is $\varphi(\theta, x) = (\theta, r_\theta(x))$ for every $(\theta, x)\in S^1\times S^2$ and $x\mapsto r_{\theta}(x)$ is the nontrivial element in $\pi_1(\So(3)) = \Z/2$  \cite[Section 6.1]{[Akbulut]}. 

\begin{theorem}\label{Theorem Akbulut}Akbulut \cite{[Akbulut2]}. There is a smoothly embedded 2-sphere\begin{equation}\label{2-sphere}\Sigma\hookrightarrow (S^3\widetilde{\times} S^1)\cs(S^2\times S^2)\end{equation}such that performing a Gluck twist along it yields a four-manifold $A$ that is homeomorphic, but not diffeomorphic to $(S^3\widetilde{\times} S^1)\cs(S^2\times S^2)$ and that satisfies the following properties.
\begin{itemize}
\item Torres \cite{[Torres]}. The value of the $\eta$-invariant for both $\Pin^+$-structures $(A, \pm \phi_A)$ is $\eta(A, \pm \phi_A) = 1 \mod 2\Z$.
\item Torres \cite{[Torres]}. The orientation 2-cover of $A$ is diffeomorphic to the connected sum $(S^1\times S^3)\cs2(S^2\times S^2)$.
\end{itemize}
\end{theorem}



\subsection{Construction of a four-manifold with prescribed fundamental group}\label{Section Standard}We occupy ourselves with a proof of the following theorem in this section. The deficiency of a group $G$ is denoted by $\Def(G)$.

\begin{theorem}\label{Theorem 1} Let $G$ be a finitely presented group that has a subgroup of index two. There is a closed smooth non-orientable four-manifold $M(G)$ with Euler characteristic $\chi(M(G)) = 2 - 2\Def(G)$, second Stiefel-Whitney class $w_2(M(G)) = 0$, whose fundamental group is $\pi_1(M(G)) \cong G$ and such that $[(M(G), \phi_{M(G)})] = 0\in \Omega^{\Pin^+}_4$ for every $\Pin^+$-structure $(M(G), \phi_{M(G)})$. 
\end{theorem}

A statement similar to Theorem \ref{Theorem 1} appears in \cite[p. 253]{[Kreck]}. 

\begin{proof}We are indebted to Daniel Kasprowski for suggesting the following argument. Let $K(G)$ be a 2-complex with $\pi_1(K(G)) = G$ associated to a presentation of $G$ that realizes its deficiency so that its Euler characteristic is $\chi(K(G)) = 1 - \Def(G)$. Choose $K(G)$ to have non-trivial first Stiefel-Whitney class and second Stiefel-Whitney class $w_1^2(K(G)) + w_2(K(G)) = 0$.  A result of Wall says that there is a five-dimensional thickening $Y(G)$ of $K(G)$ whose Stiefel-Whitney classes are given by $w_i(\nu(Y(G)) = w_i(K(G))$ for $i = 1, 2$ and for which there is an isomorphism $\pi_1(\partial Y(G)) \rightarrow \pi_1(K(G))$ \cite[Proposition 5.1]{[Wall]}. The boundary $\partial Y_G = M(G)$ is a closed smooth non-orientable four-manifold with fundamental group $G$, Euler characteristic $\chi(M(G)) = 2 - 2\Def(G)$ (see \cite[Proposition II.2]{[KreckSchafer]}), and vanishing second Stiefel-Whitney class $w_2(M(G)) = 0$. There is a total of $k = \Rk_{\Z/2}(H^1(M(G); \Z/2))$ (the rank of the cohomology group) $\Pin^+$-structures $\{(M(G), \phi^i_{M(G)}): i = 1, \ldots, 2^k\}$  \cite[Corollary 6.4]{[Stolz]}. Each of these $\Pin^+$-structures arises from a $\Pin^+$-structure on $Y(G)$. Thus, $[(M(G), \phi^i_{M(G)})] = 0\in \Omega^{\Pin^+}_4$ for every $1\leq i \leq 2^k$.\end{proof}

\subsection{Two constructions of the four-manifold $M_G$ of Theorem \ref{Theorem A}}The four-manifold of Theorem \ref{Theorem A} is defined to be\begin{equation}\label{Construction M}M_G: = M(G)\cs (S^2\times S^2).\end{equation}Every $\Pin^+$-structure $(M_G, \phi_{M_G})$ arises from $\Pin^+$-structures\begin{center}$(M(G), \phi_{M(G)})$ and $(S^2\times S^2, \phi_{S^2\times S^2})$,\end{center}and we have that\begin{equation}\eta(M_G, \phi_{M_G}) = 0 \mod 2\Z\end{equation}and\begin{equation*}[(M_G, \phi_{M_G})] = [(M(G), \phi)] = 0\in \Omega^{\Pin^+}_4\end{equation*} by construction. A slight modification of the assemblage (\ref{Construction M}), which yields a useful expression of the four-manifold $M_G$ is as follows.  Let $\alpha_0\subset (S^3\widetilde{\times} S^1)\cs(S^2\times S^2)$ be a simple loop whose homotopy class generates the group $\pi_1((S^3\widetilde{\times} S^1)\cs(S^2\times S^2)) = \Z$. Let $\gamma \subset M(G)$ be an orientation-reversing simple loop whose homotopy class is a generator of $\pi_1(M(G)) = G$. The tubular neighborhoods of the two loops $\alpha$ and $\gamma$ are diffeomorphic to $D^3\widetilde{\times} S^1$ and\begin{equation*}(S^3\widetilde{\times} S^1)\cs(S^2\times S^2)\setminus \nu(\alpha_0) = (D^3\widetilde{\times} S^1)\cs(S^2\times S^2).\end{equation*} We can glue these two four-manifolds along these loops to produce\begin{equation}\label{New Construction 1}M_G  = (M(G)\setminus \nu(\gamma))\cup ((D^3\widetilde{\times}S^1)\cs(S^2\times S^2))= M(G)\cs(S^2\times S^2).\end{equation}




\subsection{Construction of the four-manifold $N_G$ of Theorem \ref{Theorem A}, its orientation 2-cover and the existence of a homeomorphism}\label{Section Exotic}Let $A$ be the four-manifold of Theorem \ref{Theorem Akbulut} constructed by Akbulut that is homeomorphic, but not diffeomorphic to $(S^3\widetilde{\times} S^1)\cs (S^2\times S^2)$ \cite[Theorem 1]{[Akbulut1]}. Let $\alpha\subset A$ be a simple loop whose homotopy class generates the fundamental group $\pi_1(A) = \Z$. Let $\gamma \subset M(G)$ be an orientation-reversing simple loop whose homotopy class is a generator of $\pi_1(M(G)) = G$, where $M(G)$ is the closed smooth non-orientable 4-manifold of Theorem \ref{Theorem 1}. Build\begin{equation}\label{Construction N}N_G: = (M(G)\setminus \nu(\gamma))\cup (A\setminus \nu(\alpha))\end{equation}such that the induced $\Pin^+$-structures induced on the boundary components match. There are homeomorphisms\begin{equation}\label{Homeo 1}N_G\rightarrow M(G)\cs(S^2\times S^2)\rightarrow M_G\end{equation}by construction given that the four-manifold $A$ is homeomorphic to the connected sum $(S^3\widetilde{\times} S^1)\cs(S^2\times S^2)$. Moreover, the orientation 2-cover $\widehat{N_G}\rightarrow N_G$ is diffeomorphic to orientation 2-cover of (\ref{New Construction 1}) given that the orientation 2-cover of $A$ is diffeomorphic to $(S^1\times S^3)\cs2(S^2\times S^2)$ by the third clause of Theorem \ref{Theorem Akbulut}. 

The existence of a smoothly embedded 2-sphere $S\hookrightarrow M_G$ with trivial tubular neighborhood $\nu(S) = D^2\times S^2$ such that performing a Gluck twist to it produces the 4-manifold $N_G$ follows from Theorem \ref{Theorem Akbulut}. Indeed, the 2-sphere\begin{equation}\label{2-sphere}\Sigma\hookrightarrow (S^3\widetilde{\times} S^1)\cs(S^2\times S^2)\end{equation}that was constructed by Akbulut is disjoint from the simple loop $\alpha_0\subset (S^3\widetilde{\times} S^1)\cs(S^2\times S^2)$ used in the construction of $M_G$ in (\ref{New Construction 1}) by a general position argument: the sum of the dimensions of these two submanifolds is strictly less than the dimension of the ambient manifold. It follows from the construction of $N_G$ in (\ref{Construction N}), that this four-manifold is obtained by performing a Gluck twist to a smoothly embedded 2-sphere in $M_G$.




\subsection{The $\Pin^+$-bordism class of $N_G$ and distinguishing the diffeomorphism types of $M_G\cs(n - 1)(S^2\times S^2)$ and $N_G\cs(n - 1)(S^2\times S^2)$ for $n\in \N$}\label{Section Smooth Structures}From the construction (\ref{Construction N}), it follows that there is a $\Pin^+$-bordism between $N_G$ and the disjoint union of $M(G)\sqcup A$ with their respective $\Pin^+$-structures. This observation and Theorem \ref{Theorem 1}  yield the following proposition.

\begin{proposition}\label{Proposition Bordism Classes} For every choice of $\Pin^+$-structure $(M_G, \phi_{M_G})$ on $M_G$, we have that\begin{equation}\label{Eta 1}\eta(M_G, \phi_{M_G}) = 0 \mod 2\Z\end{equation}and\begin{equation}[(M_G, \phi_{M_G})] =  0 \in \Omega^{\Pin^+}_4.\end{equation}For every choice of $\Pin^+$-structure $(N_G, \phi_{N_G})$, we have that\begin{equation}\label{Eta 2}\eta(N_G, \phi_{N_G}) = 1 \mod 2\Z\end{equation}and\begin{equation}[(N_G, \phi_{N_G})] = [(A, \pm \phi_{A})] =  8 \in \Omega^{\Pin^+}_4.\end{equation}Moreover, for every $n\in \N$ and for every choice of $\Pin^+$-structures on the connected sums\begin{center}$(M_G\cs(n - 1)(S^2\times S^2), \phi_{M_G}')$ and $(N_G\cs(n - 1)(S^2\times S^2), \phi_{N_G}')$,\end{center}we have\begin{equation}[(M_G\cs(n - 1)(S^2\times S^2), \phi_{M_G}')] = [(M_G, \phi_{M_G})] = 0\in \Omega^{\Pin^+}_4\end{equation} and\begin{equation}[(N_G\cs(n - 1)(S^2\times S^2), \phi_{N_G}')] = [(N_G, \phi_{N_G})] = 8\in \Omega^{\Pin^+}_4.\end{equation}
\end{proposition}

The value (\ref{Eta 1}) was computed in Section \ref{Section Standard}, while the value (\ref{Eta 2}) follows from the discussion in Section \ref{Section Exotic} since the $\eta$-invariant is additive under disjoint unions. These values determined the $\Pin^+$-bordism classes \cite[Introduction]{[KirbyTaylor]}. Theorem \ref{Theorem Stolz} states that the value of the $\eta$-invariant of $((n - 1)(S^2\times S^2), \phi_n)$ is\begin{equation}\eta((n - 1)(S^2\times S^2), \phi_n) = \sigma((n - 1)(S^2\times S^2)) = 0 \mod 2\Z\end{equation} for every $n\in \N$, where $\sigma((n - 1)(S^2\times S^2) = 0$ is the signature of the connected sum of $n - 1$ copies of $S^2\times S^2$.

\subsection{A proof of Theorem \ref{Theorem A}}\label{Section Blueprint}The four-manifold $M_G$ is constructed in (\ref{Construction M}), while the four-manifold $N_G$ is constructed in (\ref{Construction N}). It is straight-forward to see that $\chi(M_G) = 4 - 2\Def(G)$. The constructions (\ref{New Construction 1}) and (\ref{Construction N}) imply that $M_G$ is homeomorphic to $N_G$ as it was discussed in Section \ref{Section Exotic}, where we also argued the existence of a diffeomorphism between the orientation 2-covers of $M_G$ and $N_G$ and the Gluck twist construction of $N_G$. Proposition \ref{Proposition Bordism Classes} and Theorem \ref{Theorem Stolz Main} imply that there is no diffeomorphism $M_G\cs(n - 1)(S^2\times S^2)\rightarrow N_G\cs(n - 1)(S^2\times S^2)$ for any $n\in \N$. Given that $M_G$ and $N_G$ are related by a Gluck twist, there is a diffeomorphism between $M_G\cs\mathbb{CP}^2$ and $N_G\cs\mathbb{CP}^2$ \cite[Exercise 5.2.7 (b)]{[GompfStipsicz]}. This concludes the proof of Theorem \ref{Theorem A}.\hfill $\square$

\subsection{Properties of Kreck's examples}\label{Section Kreck}As the referee kindly pointed out, weaker versions of the fourth and sixth items of Theorem \ref{Theorem A} immediately follow from coupling Kreck's examples \cite[Theorem 1]{[Kreck]} with a result of Gompf \cite{[Gompf1]}, which says that any two closed non-orientable homeomorphic four-manifolds become diffeomorphic after taking connected sums with a sufficiently large number of copies of the complex projective plane. In this section, we show that Kreck's examples $M_1:= M_{\pi}\cs K3$ and $M_2:= M_{\pi}\cs 11(S^2\times S^2)$ also satisfy the properties listed in the fourth and sixth items of Theorem \ref{Theorem A}, independently of the choice of $M_\pi$. The existence of a diffeomorphism $\widehat{M}_1\rightarrow \widehat{M}_2$ between the orientation 2-covers follows from the existence of a diffeomorphism $K3\cs \overline{K3}\rightarrow 22(S^2\times S^2)$ \cite[Corollary 2]{[Gompf2]}. This establishes the property in the fourth item of our main result for these examples. An argument using Kirby calculus shows the existence of a diffeomorphism $K3\cs \mathbb{CP}^2\rightarrow 4\mathbb{CP}^2\cs 19\overline{\mathbb{CP}^2}$ \cite[Exercise 8.3.4(d)]{[GompfStipsicz]} (see also \cite{[Gompf2], [Mandlebaum], [Moishezon]}). Thus, the four-manifolds $M_{\pi}\cs K3\cs \mathbb{CP}^2$ and $M_{\pi}\cs 11(S^2\times S^2)\cs \mathbb{CP}^2$ are diffeomorphic to $M_{\pi}\cs 23 \mathbb{CP}^2$.\hfill $\square$



\end{document}